\documentclass[a4paper,11pt]{amsart}
\usepackage{amsfonts}
 \usepackage{enumerate}
\usepackage{amsmath, amsthm}
\usepackage{amscd}
\usepackage{amssymb}
\input xy
\xyoption{all}
\usepackage{latexsym,graphicx}
\usepackage[centering]{geometry}

\newcommand{\sspace}{\vspace{0.25cm}}

 \theoremstyle{plain}
\newtheorem{theor}{Theorem}[section]

\newtheorem{prop}[theor]{Proposition}
\newtheorem{lem}[theor]{Lemma}
\newtheorem{sublem}[theor]{Sublemma}
\newtheorem{cor}[theor]{Corollary}

\theoremstyle{remark}
\newtheorem{rem}[theor]{Remark}
\newtheorem{rems}[theor]{Remarks}
\newtheorem{Example}[theor]{Example}

\theoremstyle{plain}
\newtheorem{defi}[theor]{Definition}

\numberwithin{equation}{section}

\newcommand{\CC}{{\mathbb C}}
\newcommand{\RR}{{\mathbb R}}
\renewcommand{\SS}{{\mathbf S}}
\newcommand{\QQ}{{\mathbb Q}}
\newcommand{\FF}{{\mathcal{ F}}}
\newcommand{\ZZ}{{\mathbb Z}}
\renewcommand{\AA}{{\mathbf A}}
\newcommand{\G}{{\mathbf G}}
\newcommand{\HH}{{\mathcal H}}

\newcommand{\PP}{{\mathbf P}}
\newcommand{\MM}{{\mathbf M}}

\newcommand{\UU}{{\mathbf U}}

\newcommand{\NN}{{\mathbb N}}

\newcommand{\Ga}{\Gamma}

\newcommand{\Gr}{{\mathbf{Gr}}}

\newcommand{\ti}[1]{\mbox{$\tilde{#1} $}}

\newcommand{\ol}{\overline}

\newcommand{\wt}{\widetilde}
\newcommand{\lo}{\longrightarrow}

\newcommand{\sm}{{\rm sm}}

\newcommand{\End}{{\rm End}\,}

\newcommand{\Res}{{\rm Res}}
\newcommand{\Sh}{{\rm Sh}}

\newcommand{\ad}{{\rm ad}}

\newcommand{\GL}{{\rm \bf GL}}

\newcommand{\tr}{\textnormal{tr}}
\newcommand{\alg}{\textnormal{alg}}
\newcommand{\proj}{{\mathbb P}}
\newcommand{\diam}{{\hfill \nobreak} $\Box$}

\newcommand{\F}{\mathcal{F}}

\newcommand{\Vol}{\textnormal{Vol}}

\newcommand{\Aut}{\textnormal{Aut}}
\newcommand{\an}{\textnormal{an}}

\newcommand{\bH}{{\mathbf H}}
\newcommand{\diag}{\textnormal{diag}}

\def\Fp{\mathfrak{p}}

\def\Fg{\mathfrak{g}}

\newcommand{\cF}{\mathcal{F}}
\newcommand{\cA}{{\mathcal A}}

\newcommand{\cD}{{\mathcal D}}

\newcommand{\cR}{{\mathcal R}}

\newcommand{\cX}{{\mathcal X}}

\newcommand{\AAA}{{\mathbb A}}

\newcommand{\lto}{\longrightarrow}

\newcommand{\bF}{\mathbf{F}}

\newcommand{\SL}{{\mathbf{SL}}}

\newcommand{\Lie}{{\rm Lie}}

\newcommand{\Imm}{{\rm Im}}
\newcommand{\Ree}{{\rm Re}}
\newcommand{\mon}{\textnormal{mon}}
\newcommand{\Zar}{\textnormal{Zar}}
\newcommand{\class}{\textnormal{class}}
\newcommand{\Id}{\textnormal{Id}}

\begin{document}

\title{The hyperbolic Ax-Lindemann-Weierstra\ss\- conjecture}
\author{B. Klingler, E.Ullmo, A.Yafaev}
\thanks{Andrei Yafaev was supported by the ERC grant Project 307364 SPGSV}
\maketitle

\section{Introduction.}

\subsection{Bi-algebraic geometry and the Ax-Lindemann-Weierstra\ss\-
  property.}
Let $X$ and $S$ be complex algebraic varieties and suppose $\pi: X^\an \lo
S^\an$ is a complex analytic, {\em non-algebraic}, morphism between the associated complex
analytic spaces. In this situation the image $\pi(Y)$ of a generic
algebraic subvariety $Y \subset X$ is usually highly transcendental
and the pairs $(Y\subset X, V\subset S)$ of irreducible algebraic subvarieties such
that $\pi(Y) = V$ are rare and of particular geometric significance. We will say
that an irreducible subvariety $Y\subset X$ (resp. $V \subset S$) is
{\em bi-algebraic} if $\pi(Y)$ is an algebraic subvariety of $S$
(resp. any analytic irreducible component of $\pi^{-1}(V)$ is an
irreducible algebraic subvariety of $X$). Notice that $V\subset S$ is
bi-algebraic if and only if any analytic irreducible component of
$\pi^{-1}(V)$ is bi-algebraic.

\begin{Example} \label{ex1}
Let $\pi:= (\exp(2\pi i \cdot), \dots,\exp(2\pi i \cdot)) : \CC^n \lo
(\CC^*)^n$. One easily shows that an irreducible algebraic subvariety $Y \subset \CC^n$ (resp. $V \subset (\CC^*)^n)$) is bi-algebraic if
and only if $Y$ is a translate of a rational linear subspace of $\CC^n= \QQ^n
\otimes_\QQ \CC$ (resp. $V$ is a translate of a subtorus of
$(\CC^*)^n$). 
\end{Example}

\begin{Example} \label{ex2}
Let $\pi: \CC^n \lo A$ be the uniformizing map of a complex Abelian
variety $A$ of dimension $n$. One checks that an irreducible algebraic subvariety $V \subset A$ is bi-algebraic if
and only if $V$ is the translate of an Abelian subvariety of $A$
(cf. \cite[prop. 5.1]{UY1} for example).
\end{Example}

More generally, given $Y \subset X$ an algebraic subvariety, one may
ask for a description of the Zariski-closure $\overline{\pi(Y)}^\Zar$
of its image $\pi(Z)$. We will say that $\pi: X \lo S$ satisfy the
Ax-Lindemann-Weierstra\ss\- property if for any such $Y \subset X$ the
irreducible components of $\overline{\pi(Y)}^\Zar$ are
bi-algebraic. One checks that the Ax-Lindemann-Weierstra\ss\- property
is equivalent to the following: for any algebraic
subvariety $V \subset S$, any irreducible algebraic subvariety $Y$ of
$X$ contained in $\pi^{-1}(V)$ and maximal for this property is bi-algebraic.

\begin{Example}
In the situations of Example~\ref{ex1} and Example~\ref{ex2} Ax \cite{Ax} showed that $\pi:X
\lo S$ has the Ax-Lindemann-Weiertra\ss\- property. Namely:

- if $\pi:= (\exp(2\pi i \cdot), \dots,\exp(2\pi i \cdot)) : \CC^n \lo
(\CC^*)^n$ and $Y \subset \CC^n$ is an algebraic subvariety then any
irreducible component of $\overline{\pi(Y)}^\Zar$ is the translate of
a subtorus of $(\CC^*)^n$.

- if $\pi: \CC^n \lo A$ is the uniformizing map of a complex abelian
variety $A$ of dimension
$n$ and $Y \subset \CC^n$ is an algebraic subvariety then any
irreducible component of $\overline{\pi(Y)}^\Zar$ is the translate of
an Abelian subvariety of $A$. 
\end{Example}

\begin{rem}
Notice that Ax's theorem for $\pi:= (\exp(2\pi i \cdot), \dots,\exp(2\pi i \cdot)) : \CC^n \lo
(\CC^*)^n$ is the functional analog of the classical
Lindemann-Weierstra\ss\- transcendence theorem (\cite{L}, \cite{W})
stating that if $\alpha_1,
\dots, \alpha_n$ are $\QQ$-linearly independent algebraic numbers
then $e^{\alpha_1}, \dots, e^{\alpha_n}$ are algebraically
independent over $\QQ$. This explain our terminology.
\end{rem}

\subsection{The hyperbolic Ax-Lindemann-Weierstra\ss\- conjecture}

The main result of this paper is the proof of the
Ax-Lindemann-Weierstra\ss\- property for the uniformizing map $\pi: X
\lo S:=\Gamma \backslash X$ of any {\em arithmetic
  variety} $S$. Here $X$ denotes a Hermitian
symmetric domain and $\Gamma$ is any {\em arithmetic subgroup} 
of the real adjoint Lie group $G$ of biholomorphisms of $X$. This
means that there exists a semisimple $\QQ$-algebraic group $\G$ and a
surjective morphism with compact kernel $p: \G(\RR) \lo G$ such that $\Ga$ is commensurable with the
projection $p(\G(\ZZ))$  (cf. section~\ref{not} for the definition of
$\G(\ZZ)$ and \cite{mar} for a general reference on arithmetic
lattices). 

The Ax-Lindemann-Weierstra\ss\- property does not make sense directly
for $\pi$:  the arithmetic variety $S$ admits a natural structure of complex quasi-projective
variety via the Baily-Borel embedding \cite{BB} but the Hermitian symmetric domain $X$ is not a complex
algebraic variety. However $X$ admits a canonical
realisation  as a bounded symmetric domain $\cD \subset \CC^N$ (with
$N=\dim_\CC X$) (cf. \cite[\S II.4]{Sat3}). 

\begin{defi} \label{funddef}
We will say that a subset $Y\subset
\cD$ is an {\em irreducible algebraic subvariety} of $\cD$ if $Y$ is an
irreducible component of the analytic set $\cD \cap \wt{Y}$ where $\wt{Y}$ is an algebraic subset of $\CC^N$.
An algebraic subvariety of $\cD$ is then defined as a finite union of irreducible
algebraic subvarieties.
\end{defi}

With these definitions the
morphism $\pi$ is far from algebraic (in the simplest case where $\cD$ is the Poincar\'e disk
and $S$ is the modular curve, the map $\pi: \cD \lo S$ is 
the usual $j$-invariant seen on the disk) and it makes sense to study
the bi-algebraic subvarieties for $\pi$. In \cite{UY1} Ullmo and
Yafaev proved that the bi-algebraic subvarieties
of $S$ for $\pi$ are the {\it weakly special} ones, namely the
irreducible complex algebraic subvarieties of $S$ whose smooth locus is totally geodesic in $S$ endowed with its canonical
Hermitian metric.  

Our main result is the proof of the Ax-Lindemann-Weiertra\ss\-
property in this context:

\begin{theor} \label{AL}(The hyperbolic Ax-Lindemann-Weierstra\ss\- conjecture.)
Let $S= \Gamma \backslash \cD$ be an arithmetic variety
with uniformising map $\pi: \cD \lo S$. Let $Y\subset \cD$ be an
algebraic subvariety. Then any irreducible component of
the Zariski-closure $\overline{\pi(Y)}^\Zar$ of $\pi(Y)$ is weakly
special. 

Equivalently, let $V$ be an algebraic subvariety of $S$.
Irreducible algebraic subvarieties of $\cD$ contained in $\pi^{-1}V$
and maximal for this property are
precisely the irreducible components of the preimages of maximal weakly special 
subvarieties contained in $V$.
\end{theor}

\begin{rems}
\begin{itemize}
\item[(a)] 
The Ax-Lindemann-Weierstra\ss\- property in an
hyperbolic context was first proven by Pila in the case where $S$ is a product of
modular curves: cf. \cite[section 1.4 and theor. 6.8]{Pil}. It is a
crucial ingredient in Pila's proof of the Andr\'e-Oort conjecture for
product of modular curves. The
hyperbolic Ax-Lindemann-Weierstra\ss\- conjecture for the uniformizing
map of a general connected Shimura variety $S$ is stated in \cite[conjecture 1.2]{U2}, where Ullmo explains its role in the
proof of the Andr\'e-Oort conjecture. In \cite{UY2} Ullmo
and Yafaev prove Theorem~\ref{AL} in the special case where $S$ is
compact. In \cite{TP}, in part inspired
by \cite{UY2} and relying on \cite{PetStar},  Pila and Tsimerman proved Theorem~\ref{AL}
in the special case $S = \cA_g$, the moduli space of principally
polarised Abelian varieties of dimension $g$.
\item[(b)] Mok has a nice, entirely complex-analytic, approach to
  the hyperbolic Ax-Lindemann-Weierstra\ss\-
  conjecture. In the rank $1$ case his approach should extend some of the results of this
text to the case where $\Gamma$ is a non-arithmetic lattice. We refer to \cite{Mo1}, \cite{Mo2}
for partial results.
\item[(c)] We defined algebraic subvarieties of $X$ using the
  Harish-Chandra realisation $\cD$ of $X$ but we could have used as
  well any other {\em realisation} of $X$ in the sense of
  \cite[section 2.1]{U2}. Indeed morphisms of realisations are necessarily
  semi-algebraic, thus $X$ admits a canonical semi-algebraic structure and a canonical notion of algebraic subvarieties
  (cf. appendix~\ref{algebra} for details). Hence one can
  replace $\cD$ in Theorem~\ref{AL} by any other realisation of $X$,
  for example the Borel realisation (cf. \cite[p.52]{Mo}).
\end{itemize}
\end{rems}

\subsection{Motivation: the Andr\'e-Oort conjecture}

Let $(\G,X_\G)$ be a Shimura datum.
Let $X$ be a connected component of $X_\G$ (hence $X$ is a Hermitian symmetric domain).
We denote by $\G(\QQ)_{+}$ the stabiliser of $X$ in $\G(\QQ)$. Let $K_f$ be a compact open subgroup of 
$\G(\AAA_f)$, where $\AAA_f$ denotes the finite ad\`eles of $\QQ$ and let $\Gamma:= \G(\QQ)_{+}\cap K_f$ be the corresponding
congruence arithmetic lattice of $\G(\QQ)$.

Then the arithmetic variety $S:= \Gamma \backslash X$ is a component of the complex
quasi-projective Shimura variety
$$
\Sh_K(\G,X) := \G(\QQ)_{+} \backslash X \times \G(\AAA_f) / K_f\;\;.
$$
The variety $S$ contains the so-called special points and special
subvarieties (these are the weakly special subvarieties of $S$
containing one special point, we refer to \cite{De2} or \cite{MoMo} for the detailed definitions).
One of the main motivations for studying the
Ax-Lindemann-Weierstra\ss\-  conjecture is the Andr\'e-Oort conjecture
predicting that irreducible subvarieties of $S$ containing Zariski dense sets of special points are precisely 
the special subvarieties. The Andr\'e-Oort conjecture has been proved
under the assumption of the Generalised Riemann Hypothesis (GRH) by the authors of this
paper (\cite{UY0}, \cite{KY}).
Recently Pila and Zannier \cite{PZ} came up with a new proof of the Manin-Mumford conjecture for abelian varieties using the
flat Ax-Lindemann-Weierstra\ss\- theorem. This gave hope to prove the Andr\'e-Oort conjecture
unconditionally with the same strategy. In \cite{Pil} Pila succeeded in applying this
strategy to the case where $S$ is a product of modular curves (and more
generally, in the context of mixed Shimura varieties, when $S$ is a
product of modular curves, of elliptic curves defined over $\QQ$ and
of an algebraic torus $\G_{\textnormal{m}}^l$). Roughly speaking, the
strategy of \cite{Pil} consists of two main ingredients: the first is the problem 
of bounding below the sizes of Galois orbits of special points and the
second is the hyperbolic Ax-Lindemann-Weierstra\ss\- conjecture.
We refer to \cite{U2} for details on how the general hyperbolic
Ax-Lindemann-Weierstra\ss\- conjecture  and a good lower bound on the
sizes of Galois orbits of special points imply the full Andr\'e-Oort
conjecture. As a direct corollary of Theorem~\ref{AL} and the proof of
\cite[theor.5.1]{U2} one obtains:

\begin{cor} \label{cor}
The Andr\'e-Oort conjecture holds for $\cA_6^n$ for any positive
integer $n$.
\end{cor}

Notice also that (as explained in \cite{U2}) a new proof of the
Andr\'e-Oort conjecture under the GRH, alternative to \cite{UY0} and
\cite{KY},  
is a consequence of three ingredients: Theorem~\ref{AL}, a lower bound
under GRH for the size of 
Galois orbits of special points (provided by Tsimerman \cite{Tsi} in the
case of $\mathcal{A}_g$ and by Ullmo-Yafaev \cite{UY3} in general) and
an upper bound for the height of special points in Siegel sets. This
upper-bound has been announced by C.Daw and M.Orr \cite{DawOrr}.

\subsection{Strategy of the proof of Theorem~\ref{AL}.}

 Our general strategy for proving Theorem~\ref{AL}, which originates
  in \cite{Pil}, is also the one used in \cite{UY2} and \cite{TP}: it
  ultimately relies on estimations of rational points in
  transcendental real-analytic varieties or more generally in spaces
  definable in a o-minimal structure. Let us describe roughly this
  strategy and emphasize the new ideas involved.

(i) Let $S := \Gamma \backslash X$ and $\pi \colon X \lto S$ be the
uniformising map. Even though the map $\pi$ is transcendental, it
still enables us to relate the semi-algebraic structures on $X$ and
$S$ through a larger o-minimal structure. We refer to \cite{VDD},
\cite{VdM}, \cite[section 3]{UY2} for details on o-minimal structures.
Recall that a fundamental set  for the action of 
$\Gamma$ on $X$ is a connected open subset $\cF$ of $X$
such that $\Gamma \overline{\cF}= X$ and such that the set
$
\{\gamma\in \Gamma \ \vert \gamma\cF \cap \cF \neq \emptyset\}
$
is finite. Our first result of independent interest is the following:

\begin{theor} \label{fset}
There exists a semi-algebraic fundamental set $\cF$ for the action of $\Gamma$ on $X$
such that the restriction $\pi_{|\cF} \colon \cF \lto S$ is definable in the o-minimal structure $\RR_{\an,\exp}$.
\end{theor}

\begin{rems}
\begin{itemize}
\item[(a)]
The special case of Theorem~\ref{fset} when $S$ is compact is easy and was proven in \cite[Prop.4.2]{UY2}. In this case, the map $\pi_{|\cF}$ is even definable in $\RR_{\an}$.
Theorem~\ref{fset} in the case where $X= \HH_g$ is the  Siegel upper
half plane of genus $g$ was proven by Peterzil and Starchenko (see \cite{PetStar} and
\cite{PetStar1}): in this case they use an explicit description for
$\pi$ in terms of $\theta$-function and delicate computations with
these. Their result is a crucial ingredient in \cite{TP}. Notice
moreover that this particular
case implies Theorem~\ref{fset} for any special subvariety $S$ of
$\cA_g$ (see Proposition 2.5 of \cite{U2}). 
\item[(b)] On the other hand Peterzil and Starchenko's method does not generalize to general arithmetic varieties, where an
explicit description of $\pi$ is not available. Moreover, while the
definability of $\pi$ restricted to $\cF$ is of geometric essence, the geometric
meaning of computations with $\theta$-functions is difficult to follow. On the contrary
our general proof of Theorem~\ref{fset} is completely geometric: it
  relies on the general theory of toroidal compactifications of arithmetic
  varieties (cf. \cite{AMRT}). In particular it does not use
  \cite{PetStar} or \cite{PetStar1}.
\end{itemize}
\end{rems}

(ii) Choose a semi-algebraic fundamental set $\cF$ for the action of
$\Gamma$ as in the Theorem~\ref{fset} above. The choice of a reasonable representation $\rho: \G
\lto \GL(E)$ (cf. section~\ref{not}) allows us to define a {\em height function}
$H: \Ga \lo \RR$ (cf. definition~\ref{height}).
In section~\ref{growth} we show the following result, which is the
most original part of the proof (it mixes the geometry of toroidal
compactifications and various arguments from hyperbolic geometry, like theorem~\ref{lowerbound} of Hwang-To):

\begin{theor} \label{gro}
Let $Y$ be a positive dimensional irreducible algebraic subvariety of $X$. Define $$N_Y(T) = |\{ \gamma \in \Ga : H(\gamma)\leq T, \,Y\cap \gamma \cF \not= \emptyset \}| \;\;.$$
Then there exists a positive constant $c_1$ such that for all positive
real number $T$ large enough:
$$
N_Y(T) \geq T^{c_{1}} \;\;.
$$
\end{theor}

\begin{rem}
When $S$ is compact Ullmo and Yafaev proved in \cite[theor. 2.7]{UY2} a more
refined result. Indeed let $F:= \{\gamma \in \cF, \; \gamma
\overline{\cF} \cap \overline{\cF} \not= 0\}$ be a finite symmetric set of generators
for $\Ga$ and let $l: \Gamma \lo \NN$ be the word length
function on $\Gamma$ associated to $F$. Then Ullmo and Yafaev show that the
function $N_Y(n):= \left | \{ \gamma \in \Gamma, \; \dim (\gamma \cF
  \cap Y) = \dim Y \; \text{and}\; l(\gamma)\leq n\} \right |$ grows
  exponentially with $n\in \NN$ and Theorem~\ref{gro} follows in this
case. We were not able to obtain such a result in the general case.
\end{rem}

(iii) In section~\ref{stabil}, applying the counting result above and
some strong form of Pila-Wilkie's theorem \cite{PW}, we prove:
\begin{theor} \label{stab}
Let $V$ be an algebraic subvariety of $S$ and $Y$ a maximal
irreducible algebraic
subvariety of $\pi^{-1}V$. Let  $\Theta_Y$ denotes the stabiliser of $Y$
in $\G(\RR)$ and  define $\bH_Y$ as the connected component of the
identity of the Zariski closure of
$\G(\ZZ) \cap \Theta_Y$.
Then $\bH_Y$ is a non-trivial $\QQ$-subgroup of $\G$, such that
$\bH_Y(\RR)$ is non-compact.
\end{theor}

(iv)
Without loss of generality one can assume that $V$ is the smallest
algebraic subvariety of $S$ containing $\pi(Y)$. With this assumption
we show in section~\ref{final} that $\wt{V}$ is invariant
under $\bH_Y(\QQ)$, where $\wt{V}$ is an analytic irreducible
component of $\pi^{-1}V$ containing $Y$, and then conclude that
$\pi(Y)=V$ is weakly special 
using monodromy arguments.

\section{Notations} \label{not}
In the rest of the text:
\begin{itemize}
\item $X$ denotes a Hermitian symmetric domain (not necessarily irreducible).
\item $G$ is the adjoint semi-simple real algebraic group, whose set
  of real points, also denoted by $G$, is the group of biholomorphisms of $X$; hence $X= G/K$ where
  $K$ is a maximal compact subgroup of $G$.
\item $\Ga \subset G$ is an arithmetic lattice. This means (cf. \cite{mar}) that there exists a semi-simple linear algebraic group $\G$
over $\QQ$ and $p:\G(\RR) \lo G$ a surjective morphism with compact kernel such that $\Ga$ is
commensurable with $p(\G(\ZZ))$. Here we recall that two subgroups of a group
are commensurable if their intersection is of finite index in both of
them; moreover $\G(\ZZ)$ denotes $\G(\QQ)
\cap \rho^{-1}(\GL( E_\ZZ))$ for some faithful representation $\rho: \G \lto
\GL(E)$, where $E$ is a finite-dimensional $\QQ$-vector space and $E_\ZZ$ is
a $\ZZ$-lattice in $E$; the commensurability of $\Ga$ and $p(\G(\ZZ))$ is
independant of the choice of $\rho$ and $E_\ZZ$.
\item We denote by $n$ the dimension of $E$ as a $\QQ$-vector space.
\item 
One easily checks that Theorem~\ref{AL} holds for $\Ga$ if and only if
it holds for any $\Ga'$ commensurable with $\Ga$.
In particular without loss of generality one can and will assume that the group $\G(\ZZ)$ is neat (meaning that for any $\gamma \in \G(\ZZ)$ the group
generated by the eigenvalues of $\rho(\gamma)$ is torsion-free) and
the group $\Ga$ coincides with $p(\G(\ZZ))$ (hence is torsion-free).

\item Without loss of generality we can and will assume 
that {\em the group $\G$ is of adjoint type}. Indeed let $\lambda \colon \G \lto \G^\ad$ denotes the natural algebraic morphism to the adjoint group
 $\G^\ad$ of $\G$ (quotient by the centre).  As the Lie group $G$ is adjoint the morphism
 $p:\G(\RR) \lo G$ factorises through
\begin{equation*}
\xymatrix{
\G(\RR) \ar[r]^\lambda \ar[dr]_{p} & \G^\ad(\RR) \ar[d]^{p^\ad} \\
&G}
\end{equation*} and $\Gamma$ is commensurable with $p^\ad(\G^{\ad}(\ZZ))$.

\item Without loss of generality we can  and will assume that {\em each
  $\QQ$-simple factor of $\G$ is $\RR$-isotropic}. Indeed let
  $\bH$ be the quotient of $\G$ by its $\RR$-anisotropic $\QQ$-factors. Again, the morphism
  $p:\G(\RR) \lo G$ factorises through $\bH(\RR)$ and $\Ga$ is
  commensurable with the projection of $\bH(\ZZ)$.
\item The group $K_{\infty}:= p^{-1} K$ is a maximal compact subgroup of $\G(\RR)$. Hence
  $X= \G(\RR)/K_{\infty}$. We denote by $x_0$ the base-point $eK_{\infty}$ of $X$.

\item The quotient $S:= \Gamma \backslash X$ is a smooth complex
  quasi-projective variety. We denote by $\pi: X \lo S$ the
  uniformization map.
\item We choose $\| \cdot \|_\infty: E_\RR \lo \RR$
a Euclidean norm which is $\rho(K_\infty)$-invariant.
\item We denote by $\cX$ any realization of $X$ (cf. appendix~\ref{algebra}).
\end{itemize}

\section{Compactification of arithmetic varieties} \label{compact}

\subsection{Siegel sets} First we recall the definition of Siegel
sets for $\Ga$. We refer to \cite[\S 12]{bor} for details. We follow Borel's conventions,
except that for us the group $G$ acts on $X$ on the left.

Let $\PP$ be a minimal $\QQ$-parabolic subgroup of $\G$ such that
$K_{\infty} \cap \PP(\RR)$ is a maximal compact subgroup of
$\PP(\RR)$. Let $\UU$ be the unipotent radical of $\PP$
and let $\AA$ be a maximal split
torus of $\PP$. We denote by $\SS$ a maximal split torus of $\GL(E)$
containing $\rho(\AA)$. We denote by $\MM$ the maximal anisotropic subgroup of
the connected centralizer ${\mathbf Z}(\AA)^0$ of $\AA$ in $\PP$ and
by $\Delta$ the set of positive simple roots of $\G$
with respect to $\AA$ and $\PP$. We denote by $A\subset \SS(\RR)$ the real torus
$\AA(\RR)$. For any real number $t>0$ we let
$$
A_t := \{a \in A \; | \; a^\alpha \geq t \; \textnormal{for any}
\;\alpha \in \Delta 
\}\;\;.
$$

A Siegel set  for $\G(\RR)$ for the data
$(K_\infty, \PP, \AA)$ is a product:
$$ \Sigma'_{t, \Omega} : = \Omega \cdot A_t \cdot K_{\infty} \subset \G(\RR)\;$$
where $\Omega$ is a compact neighborhood of $e$ in $\MM^0(\RR) \cdot
\UU(\RR)$.

The image
$$
\Sigma_{t, \Omega} : = \Omega \cdot A_t \cdot x_o\subset \cX
$$
of $\Sigma'_{t, \Omega}$ in $\cX$ is called a Siegel set in $\cX$.

\begin{theor} \cite[theor.13.1]{bor} \label{siegel}
Let $X$, $G$, $\G$, $\Ga$, $\PP$, $\AA$, $K_{\infty}$, and $\cX$ be as above.
Then for any Siegel set $\Sigma_{t, \Omega} $, 
the set $\{\gamma\in \Gamma \; \vert \; \gamma\Sigma_{t, \Omega} \cap\Sigma_{t, \Omega} \neq \emptyset\}$
is finite. 
There exist a Siegel set (called a Siegel set for $\Gamma$)
$\Sigma_{t_{0}, \Omega} $ and a finite subset $J$ of $\G(\QQ)$ such that
$\cF:= J \cdot\Sigma_{t_{0}, \Omega} $ is a fundamental set 
for the action of $\Gamma$ on $\cX$.
\end{theor}

When $\Omega$ is chosen to be semi-algebraic the Siegel set 
$\Sigma_{t, \Omega}$ and the fundamental set $\cF$ are semi-algebraic as by definition of a complex
realisation (cf. appendix~\ref{algebra}) 
the action of $\G(\RR)$ on $\cX$ is semi-algebraic and  the subset 
$\Omega \cdot A_t$ of $\G(\RR)$ is semi-algebraic.

{\em We will
only consider semi-algebraic Siegel sets in the rest of the text.}

\subsection{Boundary components}  General references for this section
and the next one are \cite{MuMu} and \cite{AMRT}.

Let $\cD \hookrightarrow \CC^N$ be the Harish-Chandra realisation of
$X$ as a bounded symmetric domain.
The action of $G$ extends to the closure $\ol{\cD}$ of $\cD$ in $\CC^N$. The boundary
$\partial \cD:= \ol{\cD}\backslash \cD$ is a smooth manifold which
decomposes into a (continuous) union of {\em boundary components},
which are defined as maximal complex analytic submanifolds of $\partial
\cD$ (or alternatively as holomorphic path components of $\partial
\cD$). Explicitly, let us say that a real affine hyperplane $H
\subset \CC^N$ is a supporting hyperplane if $H \cap \ol{\cD}$ is
nonempty but $H \cap \cD$ is empty. Let $H$ be a supporting hyperplane
and let $\ol{F} = H \cap \ol{\cD} = H \cap \partial \cD$. Let $L$ be
the smallest affine subspace of $\CC^N$ which contains $\ol{F}$. Then
$\ol{F}$ is the closure of a nonempty open subset $F \subset L$ which
is then a single boundary component of $\cD$ (cf. \cite[\S III.8.11]{Sat3}). The boundary component $F$ turns out to be a bounded
symmetric domain in $L$. 

Fix a boundary component $F$. The normaliser $N(F):= \{ g \in G \; \vert \; gF
= F\}$ turns out to be a proper parabolic subgroup of $G$. The Levi
decomposition $N(F) = R(F) \cdot W(F)$ (where $W(F)$ denotes the
unipotent radical of $N(F)$ and $R(F)$ is the unique reductive Levi
factor stable under the Cartan involution corresponding to $K$) can be
refined into
\begin{equation} \label{stabilizer}
N(F) = (G_h(F) \cdot G_l(F) \cdot M(F)) \cdot V(F) \cdot U(F)
\;\;,
\end{equation}
where:

- $U(F)$ is the centre of $W(F)$. It is a real vector space;

- $V(F) = W(F) /U(F)$ turns out to be abelian. It is a real vector
space of even dimension $2l$, and we get a
decomposition $W(F) = V(F) \cdot U(F)$ using ``exp'';

- $G_l(F) \cdot M(F) \cdot V(F) \cdot U(F)$ acts trivially on $F$ and
$G_h(F)$ modulo a finite center is $\Aut^0(F)$;

- $G_h(F) \cdot M(F) \cdot V(F) \cdot U(F)$ commutes with $U(F)$ and
$G_l(F)$ modulo a finite central group acts faithfully on $U(F)$ by
inner automorphisms;

- $M(F)$ is compact.

The boundary component $F$ is said to be {\em rational}  if
$\Gamma_F:= \Ga \cap
N(F)$ is an arithmetic subgroup of $N(F)$. There are only finitely
many $\Gamma$-orbits of rational boundary components, we choose
 representatives $F_1, \dots, F_r$ for these $\Ga$-orbits. 
Then the Baily-Borel compactification of $S$
is
$$
\overline{S}^{BB}=S \cup \bigcup_{i=1}^{r} (\Gamma_{F_{i}}\backslash F_{i})
$$
with a suitable analytic structure.

\subsection{Toroidal compactifications and local coordinates}
Let $X^\vee$ be the compact dual of $X$ and $\cD \hookrightarrow
X^\vee$ be the Borel embedding. Recall that $X^\vee$ has an algebraic
action by $G_\CC$. Given a boundary component $F$ of $\cD$ we define,
following \cite[section 3]{MuMu}, an open subset $\cD_F$ of $X^{\vee}$ containing $\cD$ as follows:
$$
\cD_F = \bigcup_{g \in U(F)_{\CC}} g \cdot \cD\;\;.
$$
The embedding of $\cD$ in $\cD_F$ is Piatetskii-Shapiro's realisation of
$\cD$ as Siegel Domain of the third kind. In fact there is a canonical
holomorphic isomorphism (we refer to the proof of Lemma~\ref{def_j}
for a precise description of this isomorphism):
$$
\cD_F \stackrel{j}{\simeq} U(F)_{\CC} \times \CC^l \times F \;\;.
$$
This biholomorphism defines complex  coordinates $(x, y, t)$ on $\cD_F$, such that 
$$
\cD\stackrel{j}{\simeq} \{ (x,y,t)  \in U(F)_{\CC} \times \CC^l \times
F \; | \; \Imm(x) + l_t(y,y) \in C(F)\}\subset \cD_{F}
$$
where $ \Imm(x)$ is the imaginary part of $x$, $C(F) \subset U(F)$ is
a self-adjoint convex cone homogeneous under the $G_l(F)$-action on
$U(F)$ and $l_t \colon \CC^l \times \CC^l \lto U(F)$
is a symmetric $\RR$-bilinear form varying real-analytically with $t \in F$.
The group $U(F)_{\CC}$ acts on $\cD_{F}$ and 
in these coordinates the action of $a\in U(F)(\CC)$ is given by:
$$
(x,y,t) \lto (x+a , y, t).
$$

From now on we fix a $\Gamma$-admissible collection of polyhedra
$\boldsymbol{\sigma}=(\sigma_{\alpha})$ (cf. \cite[definition 5.1]{AMRT})
such that the associated toroidal
compactification $\overline{S} = \overline{S}_{\boldsymbol{\sigma}}$ constructed in
\cite{AMRT} is smooth projective and the complement $\ol{S} \setminus
S$ is a divisor with normal crossings. We refer to \cite{AMRT}
for details and we just recall what is needed for our purposes.

The compactification $\ol{S}$ is covered by a finite set of
coordinates charts constructed as follows (cf. \cite[p.255-256]{MuMu}):

(a) Take a rational boundary component $F$ of $\cD$;

(b) We may choose some complex coordinates $x=(x_{1},\dots,x_{k})$ on $U(F)_\CC$
(depending on the choice of $\boldsymbol{\sigma}$) such that the following diagram
commutes:
\begin{equation} \label{exp}
\xymatrix{
\cD \ar[d] \ar@{^(->}[r] & \cD_F \stackrel{j}{\simeq} U(F)_\CC \times \CC^l \times F
\ar[d]^{\exp_F}  & \\
\exp_F(\cD) = \Gamma\cap U_{F}\backslash \cD \ar[d]_{\pi_F} \ar@{^(->}[r] &
\CC^{*k}\times \CC^l \times F \ar@{^(->}[r] & \CC^{k}\times \CC^l \times F \\
S & &
}
\end{equation}
where 
$
\exp_{F}: U(F)_\CC \times \CC^l \times F \rightarrow \CC^{*k}\times \CC^l \times F
$
is given by 
\begin{equation} \label{exp_F}
(x,y,t)\mapsto (\exp(2i \pi x),y,t), \; \textnormal{where} \; \exp(2i \pi
  x)=(\exp(2i \pi x_{1}),\dots,\exp(2i \pi x_{k}))\;\;.
\end{equation}

(c) Define the ``partial compactification of $\exp_F(\cD)$ in the
direction $F$'' to be the set $\exp_F(\cD)^{\vee}$ of 
points $P$ in $\CC^{k}\times \CC^l \times F$ having a neighborhood $\Theta$
such that 
$$
\Theta\cap \CC^{*k}\times \CC^l \times F\subset \exp_F(\cD) \;\;.
$$
Then there exists an integer $m$, 
$1\leq m\leq k$, such that $\exp_F(\cD)^{\vee}$
contains
$$
S(F,\boldsymbol{\sigma})=\cup_{i=1}^{m}\{ (z,y,t)\vert z=(z_{1},\dots,z_{k}), z_{i}=0\}.
$$

(d) The basic property of $\ol{S}$ is that the covering map
$\pi_F:\exp_F(\cD) \rightarrow S$ extends to a local
homeomorphism
$
\overline{\pi_F}: \exp_F(\cD)^{\vee}\rightarrow \overline{S}
$
making the diagram
\begin{equation} \label{factor}
\xymatrix{
\cD \ar[d]^{\exp_F}  \ar@/_4pc/[dd]_{\pi}& \\
\exp_F(\cD) \ar[d]^{\pi_F}
\ar@{^(->}[r] & \exp_F(\cD)^\vee \ar[d]^{\ol{\pi}_F} \\
S  \ar@{^(->}[r] & \ol{S}
}
\end{equation}
commutative.
Moreover every point $P$ of $\overline{S}-S$
is of the form $\overline{\pi}_F((z,y,t))$
with $z_{i}=0$ for some $i\leq m$, for some $F$.

The following proposition summarizes
what we will need:
\begin{prop}\label{OSiegel}
Let $\Sigma=\Sigma_{t,\Omega} \subset \cD$ be a Siegel set for the action of $\Gamma$.
Then $\Sigma$ is covered by a finite number of open subsets $\Theta$ having the 
following properties. For each $\Theta$ there is a rational boundary component
$F$, a simplicial cone $\sigma\in \boldsymbol{\sigma}$ with  $\sigma\subset \overline{C(F)}$, a point 
$a\in C(F)$,  relatively compact subsets $U'$, $Y'$ and $F'$ of
$U(F)$, $\CC^{l}$ and $F$ respectively such that the set $\Theta$ is of the form
\begin{equation*}
\begin{split}
\Theta &\stackrel{j}{\simeq} \{ (x,y,t) \in U(F)_{\CC}\times \CC^{l}\times F  ,  \, \Ree (x)\in U', y \in Y', t \in F' \;
\vert \; \Imm(x)+l_{t}(y,y)\in \sigma+a\} \\
& \subset U(F)_{\CC}\times \CC^{l}\times F \stackrel{j^{-1}}{\simeq} \cD_F \;\;.
\end{split}
\end{equation*}
\end{prop}

\begin{proof}
Let us provide a proof of this proposition, essentially stated without
proof in \cite[p.259]{MuMu}.
Let $\cD \stackrel{\Psi}{\simeq} W(F) \times C(F) \times F$ be the real-analytic
isomorphism deduced from the group-theoretic
isomorphism~(\ref{stabilizer}) constructed in
\cite[p.233]{AMRT}. Following \cite[p.266, corollary of
proof]{AMRT}, the Siegel set $\Sigma$ is covered by a finite
number of sets $\Theta$ of the form
$$ \Theta \stackrel{\Psi}{\simeq} \omega_F \times (C_0 \cap \sigma_\alpha^F) \times E\;\;,$$
where $E\subset F$ and $\omega_W \subset W(F)$ are compact, $C_0
\subset C(F)$ is a rational core and $\sigma^F_\alpha$ is one of the
polyhedra in our decomposition of $C(F)$. 

Considering $C(F)$ as a cone in $\sqrt{-1}\cdot U(F)$ and decomposing $W(F)$ as $U(F) \cdot V(F)$, the isomorphism $\Psi$
extends to the real-analytic isomorphism $\cD_F \stackrel{\Psi}{\simeq} U(F)_\CC
\times V(F) \times F$ constructed in \cite[p.235]{AMRT}. Hence the Siegel set $\Sigma$ is covered by a finite
number of sets $\Theta$ of the form
\begin{equation} \label{exp_Psi}
 \Theta \stackrel{\Psi}{\simeq} \Psi(\cD) \cap \{(x, s, t) \in U(F)_\CC
  \times V(F) \times F \quad | \quad \Ree(x) \in U', s \in S', t \in
  F' \}
\end{equation}
where $F'\subset F$, $U'\subset U(F)$ and $S' \subset V(F)$ are
relatively compact. 

Using the definition of $j$ given in \cite[\S7]{WK} and recalled in the
proof of Lemma~\ref{def_j} below, it follows, as stated in
\cite[p.238]{AMRT}, that the diffeomorphism $j \circ \Psi^{-1}: U(F)_\CC \times V(F) \times F \simeq U(F)_\CC \times
\CC^l \times F$ is a change of trivialisation of the real-analytic bundle 
$$
\xymatrix{
\cD_F \ar[d]^{\pi'_F} \ar@/_2pc/[dd]_{\pi_F}\\
\cD'_F \ar[d]^{p_F} \\
F
}
$$
studied in \cite[p.237]{AMRT}.
Here the map $\pi'_F$ is a $U(F)_\CC$-principal homogeneous space, the map
$p_F$ is a $V(F)$-principal homogeneous space, and the map $j \circ
\Psi^{-1}$ is $U(F)_\CC$-equivariant and respects the fibrations over $F$.
These two properties ensure that $j \circ \Psi^{-1}$ identifies the
set $\Psi(\Theta)$ of (\ref{exp_Psi}) to a set of
the required form 
\begin{equation*}
\begin{split}
\Theta &\stackrel{j}{\simeq} \{ (x,y,t) \in U(F)_{\CC}\times \CC^{l}\times F  ,  \, \Ree (x)\in U', y \in Y', t \in F' \;
\vert \; \Imm(x)+l_{t}(y,y)\in \sigma+a\} \\
& \subset U(F)_{\CC}\times \CC^{l}\times F \;\;.
\end{split}
\end{equation*}
\end{proof}

\section{Definability of the uniformisation map: proof of Theorem~\ref{fset}.}\label{PS}

First notice that, although the variety $S$ does not
canonically embed into some $\RR^n$, the statement of
Theorem~\ref{fset} makes sense as $S$ has a canonical structure of
real algebraic manifold, hence of $\RR_{\an, \exp}$-manifold: cf. appendix~\ref{defi}.

By Theorem~\ref{siegel} there exist a semi-algebraic Siegel set 
$\Sigma$ and a finite subset $J$ of $\G(\QQ)$ such that
$\cF:= J \cdot\Sigma $ is a (semi-algebraic) fundamental set 
for the action of $\Gamma$ on $\cD$. Hence Theorem~\ref{fset} follows from the following more
precise result.

\begin{theor}\label{PetStar}
The restriction $\pi_{|\Sigma} : \Sigma \lto S$ of the uniformising
map $\pi \colon \cD \lto S$ is definable in $\RR_{\an, \exp}$.
\end{theor}

\begin{proof}

By the Proposition~\ref{OSiegel} we know that $\Sigma$ is covered by a finite union of open subsets $\Theta$ with the 
following properties. For each $\Theta$ there is a rational boundary component
$F$, a simplicial cone $\sigma\in \boldsymbol{\sigma}$ with  $\sigma\subset \overline{C(F)}$, a point 
$a\in C(F)$,  relatively compact subsets $U'$, $Y'$ and $F'$ of
$U(F)$, $\CC^{l}$ and $F$ respectively such that the set $\Theta$ is of the form
\begin{equation} \label{descr}
\begin{split}
\Theta & \stackrel{j}{\simeq} \{ (x,y,t) \in U(F)_{\CC}\times \CC^{l}\times F,   \, \Ree (x)\in U', y \in Y', t \in F' \;
\vert \;  \Imm(x)+l_{t}(y,y)\in \sigma+a\} \\
& \subset U(F)_{\CC}\times \CC^{l}\times F \;\;.
\end{split}
\end{equation}

We first prove that the holomorphic coordinates we introduced on
$\cD_F$ are definable:
\begin{lem} \label{def_j}
The canonical isomorphism $j: \cD_F \simeq  U(F)_{\CC}\times
\CC^{l}\times F$ is semi-algebraic.
\end{lem}

\begin{proof}
The isomorphism~$j$ was studied in \cite{PS} and in full generality in
\cite[\S 7]{WK} (cf. \cite[\S 1.6]{BB} for a survey). To keep the
amount of definitions at a reasonable level we follow in this proof
(and this proof only) the notations of Wolf and Koranyi in \cite{WK}. For example our $X$, resp. $X^\vee$ is denoted by $M$, resp. $M^*$.

Let $\xi \colon \Fp^- = \CC^N \lto M^*$ be
the Harish-Chandra morphism defined by $\xi(E) = \exp(E)\cdot x$
(cf. \cite[p.901]{WK}; in the notations of Wolf and Koranyi $x$ is the
base point of $M^*$). This is a holomorphic embedding onto a dense
open subset of $M^*$. Notice that the map $\xi$ is real algebraic: indeed
$\Fp^-$ is a nilpotent sub-algebra of $\Fg^\CC$ hence the
exponential is polynomial in restriction to $\Fp^-$. The bounded
symmetric domain $\cD$ is $\xi^{-1}(G^0(x))$. 

Let $\Delta$ be a maximal set of strongly orthogonal positive non-compact roots
of $\Fg^\CC$ as in \cite[p.901]{WK}. For any $\alpha \in \Delta$ let
$c_\alpha \in G$ be the partial Cayley transform of $M$ associated to $\alpha$
(cf. \cite[p.902]{WK}, recall that with the notations of Wolf and
Koranyi $G$ is the compact form of the complexified group $\G^\CC$!). For a subset $\theta \subset \Delta$ we denote
by $c_\theta := \prod_{\alpha \in \theta} c_\alpha$ the partial
Cayley transform associated with $\theta$ (cf. \cite[\S 4.1]{WK}). 

Following \cite[theor. 4.8]{WK} there exists a unique subset $\theta
\subset \Delta$ such that $F=  \xi^{-1} c_{\Delta -
  \theta} M_\theta$, where $M_\theta =
G_\theta^0(x)$ is defined in \cite[p.912]{WK}. Let $\Fp_\theta^{-1}
\subset \Fp^-$ be defined as in \cite[p.912]{WK}, let $\Fp_{\Delta -
  \theta,1}^-$ be the $(+1)$-eigenspace of $\ad (c_{\Delta -
  \theta}^4)$ on $\Fp_{\Delta-\theta}^-$ and $\Fp_{2}^{\theta, -}$ be the $(-1)$-eigenspace of $\ad (c_{\Delta -
  \theta}^4)$ on $\Fp^-$. One has a canonical decomposition (cf. \cite[p.933]{WK} ):
\begin{equation} \label{decomp}
 \Fp^- = \Fp_{\Delta - \theta,1}^- \oplus \Fp_{2}^{\theta, -} \oplus
\Fp_\theta^- \;\;.
\end{equation}

The decomposition~(\ref{stabilizer}) of the normalizer $N(F) =B^\theta$
(cf. \cite[remark 3 p.932]{WK}) is proven in \cite[theorem
6.8]{WK}. In particular it follows that $\exp_{\Delta-\theta}:= \exp \circ \,\ad \, c_{\Delta-\theta}:
\Fp_{\Delta - \theta,1}^- \lo U(F)_\CC$ and $\exp :\Fp_{2}^{\theta, -}
\lo \CC^l$ are polynomial isomorphisms, while $F \subset \Fp^-$ is a
bounded symmetric domain of $\Fp_\theta^-$.

Following \cite[\S 7.6 and \S 7.7]{WK} the map $j: \cD \lto
U(F)_\CC \times \CC^l \times F \subset U(F)_\CC \times \CC^l
\times \Fp_\theta^- $
is the composition of the semi-algebraic holomorphic maps
$$
\xymatrix{
 \cD \ar@{->}[rr]^>>>>>>>>>{\xi^{-1} c_{\Delta - \theta} \xi} & &\Fp^- = \Fp_{\Delta - \theta,1}^- \oplus \Fp_{2}^{\theta, -} \oplus
\Fp_\theta^- \ar@{->}[rr]^>>>>>>>>{(\exp_{\Delta -\theta}, \exp, \Id)} & & U(F)_\CC \times \CC^l
\times \Fp_\theta^-
}
$$
which finishes the proof of Lemma~\ref{def_j}.
\end{proof}

The previous lemma enables us to forget about the definable biholomorphism
$j$. From now on and for simplicity of notations we simply write $\cD_F = U(F)_{\CC}\times \CC^{l}\times F$.

In the description~(\ref{descr}) we may and do assume that $U'$, $Y'$
and $F'$ are semi-algebraic subsets respectively of  $U(F)_\CC$,
$\CC^l$ and $F$. Then the set $\Theta$ is definable in $\RR_{\an}$ because:
\begin{itemize}
\item[-] the function 
$\psi:Y'\times F'\rightarrow U(F)$ defined by $\psi(y,t)=l_{t}(y,y)$ is analytic and defined
on a compact semi-algebraic set.
\item[-] the cone $\sigma$ is polyhedral, hence semi-algebraic.
\end{itemize}

Hence the restriction $\pi_{|\Sigma}: \Sigma \lo S$ is definable in $\RR_{\an,\exp}$ if and only if the
restriction $\pi_{|\Theta}: \Theta \lo S$ to any set $\Theta$ appearing in the
proposition \ref{OSiegel} is definable in $\RR_{\an,\exp}$. 

Fix such a set 
$$
 \Theta= \{ (x,y,t),  y \in Y', t \in F', \Ree(x)\in U'\vert \Imm(x)+l_{t}(y,y)\in \sigma+a\}
 $$
associated to a rational boundary component $F\in \{F_{1},\dots,
F_{r}\}$. 

Consider the left-hand side of the diagram~(\ref{factor}):
\begin{equation*} 
\xymatrix{
\cD \ar[d]^{\exp_F}  \ar@/_4pc/[dd]_{\pi}\\
\exp_F(\cD) \ar[d]^{\pi_F}\\
S  
}
\end{equation*}

Recall that 
$
\exp_{F}: \cD_{F}\rightarrow \CC^{*k}\times \CC^l \times F
$
is given by 
\begin{equation*} 
(x,y,t)\mapsto (\exp(2i \pi x ,y,t), \; \textnormal{where} \; \exp(2i \pi
  x)=(\exp(2i \pi x_{1}),\dots,\exp(2i \pi x_{k}))\;\;.
\end{equation*}
The function $\Ree(x_i)$, $1 \leq i \leq k$, is bounded on $\Theta$
hence the restriction to $\Theta$ of the map $x\mapsto \exp(2 i \pi
\Ree(x))$ is definable in $\RR_\an$. On the other hand the restriction
to $\Theta$ of the function $x\mapsto \exp(-2
\pi \Imm(x))$ is definable in $\RR_{\exp}$ by definition of $\RR_{\exp}$. Thus the
restriction to $\Theta$ of the map $\exp_{F}$ is definable in
$\RR_{\an,\exp}$ and  we are reduced to showing that $\pi_F : \exp_F(\Theta)
\lo S$ is definable in $\RR_{\an, \exp}$.

Consider the lower part of the diagram~(\ref{factor}):
\begin{equation*}
\xymatrix{
\exp_F(\cD) \ar[d]^{\pi_F}
\ar@{^(->}[r] & \exp_F(\cD)^\vee \ar[d]^{\ol{\pi}_F} \\
S  \ar@{^(->}[r] & \ol{S}\;\;.
}
\end{equation*}
As $U', V', F'$ are relatively compact and the imaginary part of $x$
has a lower bound on $\Theta$, the closure
$\overline{\exp_F(\Theta)}$ of $\exp_F(\Theta)$ is compact in $\exp_F(\cD)^\vee$. Hence $\pi_F : \exp_F(\Theta)
\lo S$, which is the restriction of the analytic
map $\overline{\pi}_F: \exp_F(\cD)^\vee \lo \overline{S}$ to the
relatively compact subset $\exp_F(\Theta)$ of $\exp_F(\cD)^\vee$, is
definable in $\RR_{\an}$.

\end{proof}

\section{Proof of Theorem~\ref{gro}} \label{growth}

\subsection{Distance, norm, height} \label{distance, norm, height}

\subsubsection{Distance}
Let $*$ be the adjunction on $E_\RR$ associated to the Hilbert
structure $\|\cdot \|_\infty$ on $E_\RR$. The restriction of the
bilinear form $(u,v) \mapsto \tr(u^* v)$ to the Lie
algebra $\Lie(\G(\RR))$ defines a $\G(\RR)$-invariant K\"ahler metric $g_X$ on
$X$. We denote by $d: X \times X \lo \RR$ the associated distance
and by $\omega$ the associated K\"ahler form.

\subsubsection{Norm} \label{norm}
We still denote by $\| \cdot \|_\infty: \End E_\RR \lo \RR$ the
operator norm associated to the norm $\|\cdot\|_\infty$ on $E_\RR$.
By restriction we also denote by $\|\cdot \|_\infty: \G(\RR) \lo \RR$ the
function $\| \cdot \|_\infty \circ \rho$.
As $K_{\infty}$ preserves the norm $\| \cdot \|_\infty$ on $E_\RR$, the function
$\|\cdot \|_\infty: \G(\RR) \lo \RR$ is $K_{\infty}$-bi-invariant, in particular
descends to a $K_{\infty}$-invariant function
$ \| \cdot \|_\infty : X  \lo \RR$.

Choose $(e_1, \dots, e_n)$ a basis of $E_\ZZ$ in which $\AA$
diagonalizes. It will be useful to compare the norm $\|\cdot\|_\infty$
with the norm $|\cdot|_\infty: \End E_\RR
\lo \RR$ defined by 
\begin{equation} \label{equanorm}
\forall \; \varphi \in \End E_\RR, \quad |\varphi|_\infty = \max_{i,j}
|\varphi_{ij}| \;\;,
\end{equation}
where $(\varphi_{ij})$ is the matrix of $\varphi$ in the basis $(e_1,
\dots, e_n)$ of $E_\RR$.

\subsubsection{Height}
\begin{defi} \label{height}
We define the (multiplicative) height function $H: \End
E_\ZZ \lo \RR$ as
$$ \forall  \,\varphi \in \End E_\ZZ, \;\; H(\varphi) = \max(1, \|\varphi\|_\infty)\;\;.$$
\end{defi}

\begin{rem}
When $\dim_\QQ E=1$, this height function coincides with the
classical multiplicative height function on rational numbers.
\end{rem}

By restriction, we also denote by $H: \G(\ZZ) \lo \RR$ the
function $H \circ \rho$.
Notice that for $\varphi \in \End E_\RR$, $\| \varphi \|_{\infty}$ is the square root of the largest
eigenvalue of the positive definite matrix $\varphi^*
\varphi$. If $\varphi \in \End E_\ZZ$ it follows that
$\|\varphi\|_\infty$ is at least $1$, hence
$$ \forall \, \varphi \in \G(\ZZ), \;\; H(\varphi) =
\|\varphi\|_\infty \geq 1\;\;.$$

We also define $H_\class$ the classical multiplicative height on $\End E$ using the
basis $(e_i^* \otimes e_j)_{i, j}$. In particular if $\varphi \in \End
E_\ZZ$ then $H_\class (\varphi) = |\varphi|_\infty$. As the
norms $\|\cdot \|_\infty$ and $|\cdot |_\infty$ are equivalent on
$\End E_\RR$ we obtain the following:

\begin{lem} \label{comp}
There exist a positive number $C$
such that
$$ \forall \varphi \in \End E_\ZZ, \;\; \frac{1}{C} \cdot H_\class(\varphi) \leq
H(\varphi) \leq C \cdot H_\class(\varphi) \;\;.$$
\end{lem}

\subsection{Comparing norm and distance}
\begin{lem} \label{compar}
For any $g \in \G(\RR)$ the following inequality holds:
$$ \log \|g\|_\infty \leq d(g\cdot x_0, x_0) \;\;.$$
\end{lem}

\begin{proof}
Let $\G(\RR) = K_\infty \cdot A_\infty \cdot K_\infty$ be a Cartan
decomposition of $\G(\RR)$ associated to $K_{\infty}$, where
$A_\infty$ is a maximal split real torus of $G$ containing $A$. Let $g \in \G(\RR)$ and write $g= k_1
\cdot a\cdot k_2$ its Cartan decomposition, with
$k_1$, $k_2$ in $K_{\infty}$ and $a \in A_\infty$. As $\|\cdot
\|_\infty$ is $K_\infty$-bi-invariant and $d$ is
$\G(\RR)$-equivariant the equalities $\log \|g\|_\infty
= \log \|a\|_\infty$ and $d(g
\cdot x_0, x_0) = d(a \cdot x_0, x_0)$ do hold. 

The torus $A_\infty$ is diagonalisable in an orthonormal basis $(f_1, \dots, f_n)$
of $E_\RR$. Write $a= \diag(a_1, \dots, a_n)$ in this basis, then:
$$ \log \|a\|_\infty = \max_{i} \log |a_i| \quad \textnormal{and}
\quad  d(a \cdot x_0, x_0) = \sqrt {\sum_{i= 1}^{n} (\log
  |a_i|)^2}$$
hence the result.
\end{proof}

\subsection{Comparing height and norms}
The main result of this section is the following:
\begin{lem} \label{compar_norm_height}
Let $\cF \subset X$
be the fundamental domain described in the Theorem~\ref{siegel}.
There exists a positive number $B$ such that:
\begin{equation} \label{comp1}
\forall \;\gamma \in
\G(\ZZ), \quad \forall \;u \in \gamma \cF, \qquad H(\gamma) \leq B
\cdot \| u \|_{\infty}^n\;\;.
\end{equation}
\end{lem}

\begin{proof}
Write $u = \gamma \cdot j \cdot x$ with $j\in J$ and $x= \omega \cdot
a \cdot k \in \Sigma'_{t_{0}, \Omega} = \Omega \cdot A_{t_{0}}\cdot
K_\infty$. Thus:
\begin{equation} \label{equa1}
u = j \cdot (j^{-1} \gamma j) \cdot a \cdot (a^{-1} \omega a) \cdot k \;\;.
\end{equation}

Notice that for each $j \in \G(\QQ)$ the groups $\G(\ZZ)$ and $j^{-1}
\G(\ZZ) j$ are commensurable (i.e. their intersection is of finite
index in both of them). As the subset $J\subset \G(\QQ)$ is finite, it
follows that the subgroup $\G(\ZZ) _J:= \G(\ZZ) \bigcap (\bigcap_{j \in
  J} j^{-1} \G(\ZZ) j)$ is of finite index in $j^{-1} \G(\ZZ) j$, $j \in J$. Choose a finite set $S$ of
representatives in $\G(\QQ)$ for the cosets $j^{-1} \G(\ZZ) j /
\G(\ZZ)_J$, $j \in \{1\} \cup J$. Hence there exists a unique $s \in S$ and $\gamma' \in
\G(\ZZ)_J \subset \G(\ZZ)$ such that $j^{-1} \gamma j = s\cdot
\gamma'$. We deduce from~(\ref{equa1}):
\begin{equation} \label{equa1'}
u = j s\cdot (\gamma '\cdot a) \cdot (a^{-1} \omega a) \cdot k \;\;.
\end{equation}

The set $J\cdot S$ is finite. The group $K_\infty$ is compact. Moreover
the set $\bigcup_{a \in A_{t_{0}}} a^{-1} \Omega
a$ is relatively compact in $G$ by \cite[Lemma 12.1]{bor}. As $\|\cdot
\|_\infty$ is sub-multiplicative, it follows 
from~(\ref{equa1'}) that there exists a positive number
$b$, depending only on $\Omega$ and $t_{0}$, such that 
\begin{equation} \label{equa2}
\|u \|_{\infty} \geq b \,\| \gamma'  \cdot a \|_\infty \;\;.
\end{equation}

As $j^{-1} \gamma j = s\cdot
\gamma'$ and $J$ and $S$ are finite sets, there exists a positive
number $b'$, depending only on $\Omega$ and $t_{0}$, such that 
\begin{equation} \label{equa2'}
\| \gamma' \|_{\infty} \geq b' \,\|\gamma \|_\infty \;\;.
\end{equation} 
Thus Lemma~\ref{compar_norm_height} follows the equality $H(\gamma) =
\| \gamma\|_{\infty}$, 
inequalities~(\ref{equa2}) and (\ref{equa2'}) and the
Sublemma~\ref{sous-lemme-majoration} below.
\end{proof}

\begin{sublem} \label{sous-lemme-majoration}
There exists a positive number $B$ depending only on
$\Omega$ and $t_0$ such that for all $\gamma \in \G(\ZZ)$ and $a \in A_{t_{0}}$ the following
inequality holds:
\begin{equation} \label{equa3}
\|\gamma\|_\infty \leq B \cdot \|\gamma \cdot a
\|_\infty^n\;\;.
\end{equation}
\end{sublem}

\begin{proof}
As the norm $\|\cdot \|_{\infty}$ on $\End E_\RR$ is
equivalent to the norm $|\cdot |_\infty$, it is enough to show
that 
$|\gamma|_\infty \leq |\gamma \cdot a
|_\infty^n$.

Let $\gamma = (\gamma_{k, l})$ be the matrix
  of $\gamma$ in the basis $(e_1, \dots , e_n)$ of $E_\ZZ$. As the torus $\AA$ is diagonalisable
  in the basis $(e_1, \dots , e_n)$, we write $a = \diag (a_1, \dots,
  a_n)$, with $a_i \in \RR^{>0}$. It follows that:
\begin{equation} \label{equa4}
\forall \; k, l \in \{1, \dots, n \}, \quad (\gamma \cdot a)_{k l} =  \gamma_{kl} \cdot
  a_l\;\;.
\end{equation}

As $\gamma$ is invertible, there
exists for each $s \in \{1, \dots, n\}$ an index $r_s \in \{1, \dots, n\}$ such that $\gamma_{r_s, s} \not =
0$. It follows from equation~(\ref{equa4}) that:
\begin{equation} \label{equa5}
\forall \; k, l \in \{1, \dots, n \}, \quad (\gamma \cdot a)_{k, l}
\cdot \prod_{s\not =
  l} (\gamma \cdot a)_{r_{s}, s} = \gamma_{k,l} \cdot \prod_{s\not =
  l} \gamma_{r_{s}, s} \cdot \prod_{s=1}^n a_s = \gamma_{k,l} \cdot \prod_{s\not =
  l} \gamma_{r_{s}, s} \;\;,
\end{equation}
where we used that $\prod_{l=1}^n a_i =1$ as $\rho(\G) \subset
  \SL(E)$.

Notice that $\Gamma = \G(\ZZ)$ hence each $\gamma_{k,l}$ is an
integer. It follows from the equation~(\ref{equa5}) that:
\begin{equation*} \label{equa6}
\forall \; k, l \in \{1, \dots, n \}, \quad  | \gamma_{k,l}| \leq  |
  \gamma_{k,l} \cdot \prod_{s\not =
  l} \gamma_{r_{s}, s} | = | (\gamma \cdot a)_{k, l} \cdot \prod_{s\not =
  l} (\gamma \cdot a)_{r_{s}, s}  | \leq (\max_{r,s} |(\gamma \cdot
a)_{r,s}|)^n.
\end{equation*}
In other words:
$
| \gamma|_\infty \leq | \gamma \cdot a |_\infty^n
$.
Hence the inequality~(\ref{equa3}) follows.

\end{proof}

\subsection{Lower bound for the volume of an algebraic curve.}
In \cite[Corollary 3 p.1227]{ht}, Hwang and To prove the following
lower bound for the area of any complex analytic curve in $\cD$~:
\begin{theor}[Hwang and To] \label{lowerbound}
Let $C$ be a complex analytic curve in $\cD$.
For any point $x_0 \in C$ there exist positive constants
$a_1, b_1$ such that for any
positive real number $R$ one has~:
\begin{equation} \Vol_C (C \cap B(x_0, R)) \geq a_1 \exp (b_1\cdot R)
  \;\;.\end{equation} 
\end{theor}
Here $\Vol_C$ denotes the area for the Riemanian metric on $C$
restriction of the metric $g_X$ on $\cD$ and $B(x_0, R)$ denotes the
geodesic ball of $\cD$ with center $x_0$ and radius $R$.

\subsection{Upper bound for the volume of algebraic curves on Siegel sets}
\begin{lem} \label{curveSiegel}
\begin{itemize}
\item[(i)]
There exists a constant $A_0>0$ such that for any algebraic curve
$C\subset \cD$ of degree $d$ we have the bound
$$ \Vol_C(C \cap \Sigma) \leq A_0 \cdot d \;\;.$$
\item[(ii)] There exists a constant $A>0$ such that for any algebraic curve
$C\subset \cD$ of degree $d$ we have the bound
$$ \Vol_C(C \cap \cF) \leq A \cdot d \;\;.$$
\end{itemize}
\end{lem}

\begin{proof}
We first prove $(i)$.
Recall that $\Sigma$ is covered by a finite union of open subsets
$\Theta$ described in Proposition~\ref{OSiegel}:  there is a rational boundary component
$F$, a simplicial cone $\sigma\in \Sigma$ with  $\sigma\subset \overline{C(F)}$, a point 
$a\in C(F)$,  relatively compact subsets $U'$, $Y'$ and $F'$ of
$U(F)$, $\CC^{l}$ and $F$ respectively such that the set $\Theta$ is of the form
$$
\Theta= \{ (x,y,t) \in \cD_\F,  y \in Y', t \in F', \Ree (x)\in U'\vert \Imm(x)+l_{t}(y,y)\in \sigma+a\}
\subset \cD_{F}=U(F)_{\CC}\times \CC^{l}\times F \;\;.
$$ 
Recall that $\omega$ denotes the natural K\"ahler form on $X$. As $C
\subset X$ is a complex analytic curve, one has:
$$  \Vol_C(C \cap \Theta) = \int_{C \cap \Theta} \omega \;\;.$$
On the other hand let $\omega_{\cD_{F}}$ be the Poincar\'e metric on
$\cD_F$ defined in the Siegel coordinates by:
$$\omega_{\cD_{F}} = \sum \frac{dx_{i}\wedge d\overline{x}_{i}}{\Imm(x_{i})^{2}}+\sum dy_{j}\wedge d\overline{y}_{j}
+\sum df_{k}\wedge d\overline{f}_{k}.
$$
Mumford \cite[Theor.3.1]{MuMu} proved that there exists a positive constant $c$
such that on $\cD$:
$$
\omega \leq c \cdot \omega_{\cD_{F}}\;\;.
$$
Hence:
$$  \Vol_C(C \cap \Theta) \leq c \int_{C \cap \Theta}
\omega_{\cD_{F}}\;\;.$$

Let $p_{x_{i}}$, $p_{y_{j}}$ and $p_{f_{k}}$ be the projections on $\cD_{F}$
to the coordinates $x_{i}$, $y_{j}$ and $f_{k}$.

As the curve $C$ has degree $d$ the restriction of these 
maps to $C\cap \Theta$ are either constant or at most $d$ to $1$,
hence 
$$  \Vol_C(C \cap \Theta) \leq c\cdot d \cdot (\sum \int_{p_{x_{i}}(\Theta)}  \frac{dx_{i}\wedge
  d\overline{x}_{i}}{\Imm(x_{i})^{2}} + \sum \int_{p_{y_{j}}(\Theta)}  dy_{j}\wedge d\overline{y}_{j}
+\sum \int_{p_{f_{k}}(\Theta)}  df_{k}\wedge d\overline{f}_{k}).$$

Let $i$ be such that the map $p_{x_{i}}$ is not constant.
In view of the description of $\Theta$ the projection
$p_{x_{i}}(\Theta)$ is contained in a usual fundamental set of the
upper-half plane, of finite hyperbolic
area.

Let $w$ be a coordinate $y_{j}$, $f_{k}$ and $p_{w}$ be the associated
projection on the $w$ axis. By the definition of $\Theta$ the
projection $p_w(\Theta)$ is a relatively compact open set of the
plane, hence of finite Euclidean area.

This finishes the proof of $(i)$.

\sspace
Let us prove $(ii)$.
As $C \cap \FF = C \cap J \cdot \Sigma$, one has the inequality:
$$\Vol_C(C \cap \cF) \leq \sum_{j \in J} \Vol_C (C \cap j\cdot \Sigma)
= \sum_{j \in J} \Vol_{j^{-1}C} (j^{-1}C \cap \Sigma) \leq |J| \cdot
A_0 \cdot d$$
where we used part $(i)$ applied to the algebraic curves $j^{-1}C$ of $\cD$,
$j \in J$,
which are of degree $d$.

This finishes the proof of Lemma~\ref{curveSiegel}.
\end{proof}

\subsection{Proof of Theorem~\ref{gro}}
Choose $C\subset Y$ an irreducible algebraic curve. To prove
Theorem~\ref{gro} for $Y$ it is enough to prove it for $C$.

Consider the set
$$ C(T):= \{z \in C \; \textnormal{and}\;\| z\|_\infty \leq T \}\;\;.$$
As $\cF$ is a fundamental domain for the action of $\Ga$ one has on
the one hand:
\begin{equation*}
\begin{split}
C(T) &= \bigcup_{\substack{\gamma \in \Gamma \\ \gamma \cF \cap C
    \not = \emptyset}} \{ u \in \gamma \cF \cap C \;
\textnormal{and}\; \|u \|_\infty \leq T \} \\
&\subset \bigcup_{\substack{\gamma \in \Gamma \\ \gamma \cF \cap C
    \not = \emptyset \\
H(\gamma) \leq B\cdot T^n}} \{ u \in \gamma \cF \cap C\} \quad
\textnormal{by Lemma~\ref{compar_norm_height}} \;\;.
\end{split}
\end{equation*}
Taking volumes: 
\begin{equation*} 
\Vol_C (C(T)) \leq \sum_{\substack{\gamma \in \Gamma \\ \gamma \cF \cap C
    \not = \emptyset \\
H(\gamma) \leq B\cdot T^n}} \Vol_C(\cF \cap \gamma^{-1} C) 
\end{equation*}
hence
\begin{equation} \label{in1}
\Vol_C (C(T)) \leq (A
\cdot d) \cdot N_C(B \cdot T^n) 
\end{equation}
where we applied Lemma~\ref{curveSiegel}(ii) to the algebraic curves
$\gamma^{-1} C$, $\gamma \in \Gamma$, which are all of degree $d$.

On the other hand if follows from Lemma~\ref{compar} that
$$ C \cap B(x_0, \log T) \subset  C(T) \;\;,$$
hence
\begin{equation} \label{in2}
\Vol_C(C \cap B(x_0, \log T)) \leq \Vol_C(C(T)) \;\;.
\end{equation}

Finally:
\begin{equation*}
\begin{split}
(A\cdot d) \cdot N_C (B\cdot T^n) & \geq \Vol_C(C(T)) \;\;\textnormal{ by inequality~(\ref{in1})} \\
& \geq \Vol_C(C
\cap B(x_0, \log T))  \;\;\textnormal{ by inequality~(\ref{in2})}\\
& \geq a_1 \exp (b_1 \log T) \;\;\textnormal{ by Theorem~\ref{lowerbound}}\;\;.
\end{split}
\end{equation*}
Hence the result.

\diam

\section{Stabilisers of a maximal algebraic subset: proof of
  Theorem~\ref{stab}.} \label{stabil}

\subsection{Pila-Wilkie theorem}
 \begin{defi}
The classical height $H_\class(x)$ of a point $x = (x_1, \dots , x_m)\in \QQ^m$ 
is defined as
$$
H_\class(x) = \max (H(x_1),\dots , H(x_m))
$$
where $H$ is the usual multiplicative height of a rational number.
\end{defi}

Let $Z\subset \RR^m$ be a subset and $T\geq 0$ a real number, we define:
$$
\Psi_\class(Z,T):= \{ x \in Z\cap \QQ^m : H_\class(x) \leq T\}
$$
and
$$
N_\class(Z,T):= |\Psi_\class(Z,T)|\;\;.
$$

For $Z\subset \RR^m$ a definable set in a o-minimal structure we
define the algebraic part $Z^\alg$ of $Z$ to be the union of all
positive dimensional semi-algebraic subsets of $Z$.

Recall (cf. definition 3.3 of \cite{UY2}), that a semi-algebraic block
of dimension $w$ in $\RR^m$ is 
a connected definable set $W\subset \RR^m$ of dimension $w$, regular
at every point, such that there exists a semi-algebraic set $A\subset
\RR^m$ of dimension $w$, regular at every point with $W\subset A$.

The following result is a strong form, proven by Pila \cite[theor.3.6]{Pil}, of the original
theorem of Pila and Wilkie \cite{PW}:

\begin{theor}[Pila-Wilkie]\label{PW}
Let $Z\subset \RR^m$ be a definable set in a o-minimal structure. 
For every 
$\epsilon > 0$, there exists a constant $C_{\epsilon} > 0$ such that
$$
N_\class(Z \backslash Z^{\alg}, T) < C_{\epsilon} T^{\epsilon}
$$
and the set $\Psi_\class(Z,T)$ is contained in the union of at most $C_{\epsilon}T^{\epsilon}$ semi-algebraic blocks.
\end{theor}

As a corollary of Theorem~\ref{PW} and Lemma~\ref{comp} one obtains:
\begin{cor} \label{PWadapte}
Let $Z\subset \End E_\RR$ be a definable set in a o-minimal structure. Define
$\Psi(Z,T):= \{ x \in Z\cap \End E_\ZZ : H(x) \leq T\}$
and
$N(Z,T):= |\Psi(Z,T)|$.
For every 
$\epsilon > 0$, there exists a constant $C_{\epsilon} > 0$ such that
$$
N(Z \backslash Z^{\alg},T) < C_{\epsilon} T^{\epsilon}
$$
and the set $\Psi(Z,T)$ is contained in the union of at most $C_{\epsilon}T^{\epsilon}$ semi-algebraic blocks.
\end{cor}

\subsection{Proof of Theorem~\ref{stab}}
Let $V$ be an algebraic subvariety of $S$ and $Y$ a maximal
irreducible algebraic
subvariety of $\pi^{-1}V$. Let $\Theta_Y$ be the stabiliser of $Y$ in $\G(\RR)$ and $\bH_Y$
be the neutral component of the Zariski-closure of $\G(\ZZ) \cap
\Theta_Y$ in $\G$. We want to show that $\bH_Y$ is a non-trivial subgroup of
$\G$, acting non-trivially on $X$.

Via $\rho : \G \hookrightarrow \GL(E)$, we view $\G(\RR)$ 
as a semi-algebraic (and hence definable) subset of $\End E_\RR$. As
$\pi_{|\cF}: \cF \lo S$ is definable by Theorem~\ref{fset}, lemmas 5.1
and 5.2 of \cite{UY2} show the following:

\begin{prop}
Let us define
\begin{equation*}
\begin{split} 
\Sigma(Y) &= \{ g\in \G(\RR) : \dim(gY \cap \pi^{-1}V \cap \cF) =
\dim(Y)\} \\
\text{and} \quad \Sigma'(Y) &= \{ g\in \G(\RR) : g^{-1}\cF \cap Y \not= \emptyset  \}.
\end{split}
\end{equation*}
The following properties hold:
\begin{enumerate}
\item The set $\Sigma(Y)$ is definable and for all $g\in \Sigma(Y)$, $gY \subset \pi^{-1}V$.
\item For all $\gamma \in \Sigma(Y)\cap \G(\ZZ)$, $\gamma Y$ is a
  maximal algebraic subset of $\pi^{-1}V$.
\item The following equality holds:
$$
\Sigma(Y)\cap \G(\ZZ)= \Sigma'(Y)\cap \G(\ZZ)\;\;.
$$
\end{enumerate}
\end{prop}

It follows that the number $N_{Y}(T)$ defined in Theorem~\ref{gro}
coincide with $|\Theta(Y,T)|$, where
$$
\Theta(Y,T):=\G(\ZZ) \cap \Psi(\Sigma(Y), T)\;\;.
$$

We can now finish the proof of the theorem \ref{stab} in exactly the same way as 
the proof of theorem 5.4 of \cite{UY2}. For the sake of completeness,
we reproduce it here. As $\Theta(Y, T) \subset \Psi(\Sigma(Y), T)$ it
follows from Corollary~\ref{PWadapte} that
for $T$ large enough, the set $\Theta(Y,T^{\frac{1}{2n}})$
is contained in at most $T^{\frac{c_1}{4n}}$ semi-algebraic blocks.
As $|\Theta(Y,T^{\frac{1}{2n}})| = N_Y(T^{\frac{1}{2n}})\geq T^{\frac{c_1}{2n}}$ by Theorem~\ref{gro}, we see that there is
a semi-algebraic block $W$ in $\Sigma(Y)$ containing at least $T^{\frac{c_1}{4n}}$ elements
$\gamma \in \Sigma(Y)\cap \G(\ZZ)$ such that $H(\gamma) \leq T^{\frac{1}{2n}}$.

Using lemma 5.5 of \cite{UY0} which applies verbatim in our case, we see that there 
exists an element $\sigma$ in $\Sigma(Y)$ such that $\sigma \Theta_Y$ contains at least $T^{\frac{c_1}{4n}}$
elements $\gamma \in \Sigma(Y)\cap \G(\ZZ)$ such that $H(\gamma)\leq T^{\frac{1}{2n}}$.

Let $\gamma_1$ and $\gamma_2$ be two elements of $\sigma \Theta_Y \cap \G(\ZZ)$ such that $H(\gamma)\leq T^{\frac{1}{2n}}$.

Let $\gamma := \gamma_2^{-1}\gamma_1 \in \G(\ZZ) \cap \Theta_Y$.
Using elementary properties of heights, we see that $H(\gamma) \leq c_n T^{1/2}$ where $c_n$ is
a constant depending on $n$ only. It follows that for all $T$ large enough, $\Theta_Y$ contains at least 
$T^{\frac{c_1}{4n}}$ elements $\gamma\in \G(\ZZ)$ with $H(\gamma)\leq T$.
Hence the connected component of the identity $\bH_Y$ 
of the Zariski closure of $\G(\ZZ) \cap \Theta_Y$ in $\G$ is a positive dimensional algebraic
subgroup of $\G$ contained in $\Theta_Y$. This finishes the proof of the theorem \ref{stab}.

\section{End of the proof of Theorem~\ref{AL}.} \label{final}

Let $V$ be an algebraic subvariety of $S$.
Our aim is to show that maximal irreducible algebraic subvarieties $Y$ of $\pi^{-1}V$ are
precisely the irreducible components of the preimages of maximal weakly special 
subvarieties contained in $V$. 

Using Deligne's interpretation of Hermitian symmetric spaces in terms
of Hodge theory the representation $\rho: \G \hookrightarrow \GL(E)$
defines a polarized $\ZZ$-variation of Hodge structure on $S$. We
refer to \cite[section 2]{MoMo} for the definition of the Hodge locus
of $X$ and $S$. Recall that an irreducible analytic subvariety $M$ of $X$
or $S$ is said to be Hodge generic if it is not contained in the
Hodge locus. If $M$ is not irreducible we say that $M$ is Hodge
generic if all the irreducible components of $M$ are Hodge generic.

Let $V' \subset V$ be the Zariski closure of $\pi(Y)$, as $Y$ is
analytically irreducible it easily follows that $V'$ is
irreducible. Replacing $V$ by $V'$ we can without loss of generality
assume that $\pi(Y)$ is not contained in a proper algebraic subvariety
of $V$. We now have to show that $\pi(Y)=V$ and $V$
is an arithmetic subvariety of $S$.

Since the group $\G$ is adjoint, it is a direct product
$$\G= \G_1 \times \dots \times \G_r$$
where the $\G_i$'s are the $\QQ$-simple factors of $\G$. This induces
decompositions
\begin{equation*}
G = \prod_{i=1}^r G_i, \quad
X = \prod_{i=1}^r X_i,\quad
\G(\ZZ) = \prod_{i=1}^r \G_i(\ZZ), \quad
\Ga = \prod_{i=1}^r \Gamma_i, \quad
S = \prod_{i=1}^r S_i,
\end{equation*}
where $G_i$ is a group of Hermitian type, $X_i$ its associated
Hermitian symmetric domain, $\Ga_i$ is an arithmetic lattice in $G_i$,
$S_i := \Gamma_i \backslash X_i$ is the associated arithmetic variety
and $\pi_i: X_i \lo S_i$ the associated uniformization map.

Our main Theorem~\ref{AL} is then a consequence of the following:
\begin{theor} \label{t1}
Let $\wt{V}$ be the an analytic irreducible component of $\pi^{-1}V$ containing $Y$.
In the situation described above, after, if necessary, reordering the factors, 
one has
$$
\wt{V} =X_1 \times \wt{V_{>1}}
$$
where $\wt{V_{>1}}$ is an analytic subvariety of $X_{2}\times \dots
\times X_r$ (in particular if $r= 1$ then $\wt{V} =X_1 =X$).
\end{theor}

We first show:

\begin{prop}
Theorem~\ref{t1} implies the main Theorem~\ref{AL}.
\end{prop}
\begin{proof}
Let $t$, $1 \leq t \leq r$, be the largest integer such that, after
reordering the factors if necessary, we have:
$$
\wt{V}  =X_1 \times \dots \times X_t \times \wt{V_{>t}}
$$
with $\wt{V_{>t}}$ an analytic irreducible subvariety of $X_{t+1}\times \dots
\times X_r$ which does not (after reordering the factors if necessary)
decompose into a product $X_{t+1} \times V_{>t+1}$.

In this case necessarily one has:
$$
Y=X_1 \times \dots \times X_t \times Y_{>t}
$$
where $Y_{>t}$ is a maximal algebraic subset of $\wt{V_{>t}}$.

Suppose that $\dim_\CC(\wt{V_{>t}}) > 0$. 
Let $x_{\leq t}$ be a special point on $X_1 \times \dots \times X_t$ 
and $x_{>t}$ be a Hodge generic point of $Y_{>t}$.
Let $\bH \subset \G$ be the Mumford-Tate group of the point $(x_{\leq
  t},x_{>t})$ of $X$ and let $X_H \subset X$ be the
$\bH(\RR)$-orbit of $x$. Replace $G$ by $H$ the group of
biholomorphisms of $X_H$, $X$ by $X_H$, $\G$ by
$\bH^\ad$, $\Gamma$ by $\Gamma_H$ the projection of $\bH(\ZZ)$ on
$H$, $S$ by $S_H:= \Gamma_H \backslash X_H$, $\pi: X \lo S$ by $\pi_H:
X_H \lo S_H$, $V$ by $V_H:= \pi_H(x_{\leq t} \times \wt{V_{>t}})$ and
$Y$ by $x_{\leq t} \times Y_{>t}$ and apply Theorem~\ref{t1} for these
new data: this shows that there exists $t' >t+1$ such that 
$\wt{V_{>t}} = X_{t+1}\times\dots \times X_{t'} \times
\wt{V_{>t'}}$. This contradicts the maximality of $t$.

Hence $\wt{V_{>t}}$ is a point $(x_{t+1}, \dots, x_r)$. Thus
$$\wt{V} = X_1 \times \dots \times X_t \times (x_{t+1}, \dots, x_r)$$
is weakly special, in particular algebraic, hence by maximality 
$$Y = \wt{V}= X_1 \times \dots \times X_t \times (x_{t+1}, \dots, x_r)$$ and $Y$ is weakly special.

\end{proof}

Let us prove theorem \ref{t1}.
Let $\bH_Y$ be the maximal connected $\QQ$-subgroup in the stabiliser of $Y$ in $\G(\RR)$. 
By Theorem~\ref{stab} the group $\bH_Y$ is a non-trivial algebraic subgroup of $\G$.

\begin{lem} \label{HYV}
The group $\bH_Y(\QQ)$ stabilises $\wt{V}$.
\end{lem}
\begin{proof}
Suppose there exists $h\in \bH_Y(\QQ)$ such that
$$
\wt{V}\not= h \wt{V}\;\;.
$$
As $Y$ is contained in $\wt{V}\cap h \wt{V}$ and $Y$ is irreducible,
we can choose an analytic irreducible component $\wt{V'}$ of $\wt{V}\cap h\wt{V}$
containing $Y$. Notice that $\pi(\wt{V'})$ is an irreducible component, say $V'$, of 
 $V \cap T_h(V)$. As $\dim_\CC(\wt{V'}) < \dim_\CC(\wt{V})$, we have that 
$\dim_\CC(V') < \dim_\CC(V)$. 

As $\pi(Y)\subset V'$, this 
contradicts the assumption that $\pi(Y)$ is Zariski dense in $V$.
\end{proof}

Choose a Hodge generic point $z$ of $V^{\sm}$ (smooth locus of $V$) and a point $\wt{z}$ 
of $\wt{V}$ lying over $z$. Let 
$$
\rho^{\mon} \colon \pi_1(V^{\sm}, z) \lto \GL(E_\ZZ)
$$
be the corresponding monodromy representation. 
We let $\Gamma_V \subset \G(\ZZ)$ be the image of $\rho$.  By usual topological Galois theory the
group $\Gamma_V$ is the subgroup of $\G(\ZZ)$ stabilising $\wt{V}$
(cf. section 3 of \cite{MoMo}), in particular $\Gamma_V$ contains
$\bH_Y(\ZZ)$.

By Deligne's monodromy theorem (see Theorem 1.4 of \cite{MoMo}), 
the connected component of the identity $\bH^\mon$ of the Zariski
closure $\overline{\Gamma_V}^{\Zar, \QQ}$ of $\Gamma_V$ in $\G$ is a normal subgroup of
$\G$. As $\G$ is semi-simple of adjoint type, after reordering the factors we may assume that 
$\bH^\mon$ coincides with $\G_1 \times \dots \times \G_t \times \{
1 \}$ for some integer $t\geq 1$. In particular $\bH_Y \subset \G_1 \times \dots \times \G_t \times \{
1 \}$.

We claim that $\Gamma_V$ normalises $\bH_Y$. Let $\gamma \in
\Gamma_V$.  Consider the $\QQ$-algebraic group $\bF$ generated by $\bH_Y$ and 
$\gamma \bH_Y \gamma^{-1}$. Then $\bF(\RR)^+ \cdot \wt{V} = \wt{V}$,
where $\bF(\RR)^+$ denotes the connected component of the identity of $\bF(\RR)$. Hence
$\bF(\RR)^+ \cdot Y \subset \wt{V}$.  By Lemma~\ref{contained} there exists an
irreducible (complex) algebraic subvariety $\ti{Y}$ of $\ti{V}$
containing $U$, hence $Y$. By maximality of $Y$ one has $\ti{Y}=Y$ hence
$$
\bF(\RR)^+ \cdot Y = Y.
$$
By maximality of $\bH_Y$, we have $\bF = \bH_Y$. This proves the claim.

As $\bH_Y$ is normalised by $\Gamma_V$, it is normalised by $\bH^\mon =\G_1 \times \dots \times \G_t
\times \{1 \}$. It follows that 
(after possibly reordering factors)
$\bH_Y$ contains $\G_1  \times \{ 1 \}$.

The fact that $\bH_Y(\RR)$ stabilises $\wt{V}$ shows (by taking the $\bH_Y(\RR)$-orbit of any point of $\wt{V}$) that $\wt{V} = X_1 \times \wt{V}_{>1}$.
This concludes the proof of Theorem~\ref{t1} and hence of Theorem~\ref{AL}.

\appendix

\section{Definability} \label{defi}

\subsection{About Theorem~\ref{fset}}
Let $\cR$ be any
fixed o-minimal expansion of $\RR$ (in our case $\cR = \RR_{\an,
  \exp}$).
Recall \cite[chap.10]{VDD} that a {\em definable manifold} of dimension $n$
is an equivalence class (for the usual relation) of triple $(X, X_i, \phi_i)_{i \in I}$ where $\{X_i : i \in I\}$ is
a finite cover of the set $X$ and for each $i \in I$:
\begin{itemize}
\item[(i)] we have injective maps $\phi_i: X_i \lto \RR^n$ such that
  $\phi_i(X_i)$ is an open, definably connected, definable set.
\item[(ii)] each $\phi(X_i\cap X_j)$ is an open definable subset of
  $\phi_i(X_i)$.
\item[(iii)] the map $\phi_{ij}: \phi_i(X_i \cap X_j) \lto \phi_j(X_i
  \cap X_j)$ given by $\phi_{ij}= \phi_j \cap \phi_i^{-1}$ is a
  definable homeomorphism for all $j \in I$ such that $X_i \cap X_j
  \not = \emptyset$.
\end{itemize}

We say that a subset $Z\subset X$ is definable (resp. open or closed)
if $\phi_i(Z \cap X_i)$ is a definable (resp. open or closed) subset
of $\phi_i(X_i)$ for all $i \in I$. A definable map between abstract
definable manifolds is a map whose graph is a definable subset of the
definable product manifold.

Notice in particular that $X=\proj^n\CC$ has a canonical structure of a
definable manifold (for any $\cR$): take $X_i = \CC^n = \{[z_o,
\dots, z_{i-1}, 1, z_{i+1}, \dots,z_n] \in \proj^n\CC\}$, $0 \leq i
\leq n$ where we identify $\CC^n$ with $\RR^{2n}$. As a corollary any complex
quasi-projective variety is canonically a definable manifold. This
apply in particular to $S$. In particular the statement of
Theorem~\ref{fset} has an intrinsic meaning.

\section{Algebraic subvarieties of $X$} \label{algebra}
Recall from \cite[section 2.1]{U2} that a realisation $\cX$
  of $X$ for $\G$ is any analytic subset of a complex quasi-projective
  variety
  $\wt{\cX}$, with a transitive holomorphic action of $\G(\RR)$ on
  $\cX$ such that for
  any $x_0 \in \cX$ the orbit map $\psi_{x_{0}} : \G(\RR) \lo \cX$
  mapping $g$ to $g\cdot x_0$ is semi-algebraic and identifies $\G(\RR)/K_{\infty}$
  with $X$. A morphism of realisations is a $\G(\RR)$-equivariant
  biholomorphism. By \cite[lemma 2.1]{U2} any realisation of $X$ has a canonical semi-algebraic
  structure and any morphism of realisations is semi-algebraic. Hence
  $X$ has a canonical semi-algebraic structure.

Let $\cX$ be a realisation of $X$ for $\G$. A subset $Y\subset
\cX$ is called an {\em irreducible algebraic subvariety} of $\cX$ if $Y$ is an
irreducible component of the analytic set $\cX \cap \wt{Y}$ where
$\wt{Y}$ is an algebraic subset of $\wt{\cX}$. By \cite[section
2]{FL} the set $Y$ has only finitely many analytic irreducible
components and these components are semi-algebraic. An {\em algebraic
  subvariety} of $\cX$ is defined to be a finite union of irreducible
algebraic subvarieties of $\cX$. 

\begin{lem} \label{lemB1}
A subset $Y$ of $\cX$ is algebraic if and only if $Y$ is a closed complex
analytic subvariety of $\cX$ and semi-algebraic in $\cX$.
\end{lem}

\begin{proof}
Let $Y \subset X$ be a closed complex analytic subvariety of $\cX$, semi-algebraic in
$\cX$. Without loss of generality we can assume that $Y$ is irreducible
as an analytic subvariety, of dimension $d$. Consider the real Zariski-closure $\wt{Y}$ of $Y$ in the real
algebraic variety $\Res_{\CC/\RR} \wt{\cX}$, where $\Res_{\CC/\RR}$ 
 denotes the Weil restriction of scalars from $\CC$ to $\RR$. 
Let us show that $\wt{Y}_\RR$ has a canonical structure of a complex
subvariety of $\wt{\cX}$. Choose an affine open cover $(\wt{\cX}_i)_{i
  \in I} \subset \AAA^{n_{i}}$ of $\wt{\cX}$ and denote by $\wt{Y}_i$ the intersection $\wt{Y}
\cap\wt{\cX}_i$. Let $i \in I$ such that $\wt{Y}_i$ is non-empty. As $Y$
  is semi-algebraic, $Y$ is open in
  $\wt{Y}$ for the Hausdorff topology, hence $Y_i:= Y \cap \wt{\cX}_i$
  is non-empty and open in
  $\wt{Y}_i$ for the Hausdorff topology. Consider the Gauss map
  $\varphi_i$ from the smooth part $\wt{Y}_i^{\sm}$ of
  $\wt{Y}_i$ to the real Grassmannian $\Gr^{2d, 2 n_i}$ of real $2d$-planes of
  $\Res_{\CC/\RR} \AAA^{n_{i}}$ associating to a point its tangent space. The map $\varphi_i$ is real analytic
  and its restriction to the open subset $Y_i^\sm$ of $\wt{Y}^\sm_i$ takes
  values in the closed real analytic subvariety $\Gr_\CC^{d, n_i} \subset \Gr^{2d, 2 n_i}$ of
  complex $d$-planes of $\AAA_\CC^{n_{i}}$. By analytic
  continuation $\varphi_i$ takes values in $\Gr_\CC^{d, n_i}$. Hence
  $\wt{Y}_i$ is a complex algebraic subvariety of $\AAA^{n_{i}}$. As
  this is true for all $i \in I$, $\wt{Y}$ is a complex algebraic
  subvariety of $\wt{\cX}$. As $Y \subset \wt{Y}$ is open and $Y$ is
 closed analytically irreducible in $\cX$, it follows that $Y$ is an
 irreducible component of $\cX \cap \wt{Y}$, hence algebraic.

The other implication is clear.
\end{proof}

As any morphism of realisations is an analytic biholomorphism and
semi-algebraic the previous lemma implies immediately:
\begin{cor}
Let $\varphi: \cX_1 \lo \cX_2$ be a morphism of realisations of
$X$. A subset $Y_1$ of $\cX_1$ is algebraic if and only if its image
$Y_2:= \varphi(Y_1) \subset \cX_2$ is algebraic. 
\end{cor}
This defines the notion of algebraic subsets of $X$.

\begin{lem} \label{contained}
Let $\cX$ be a realisation of a Hermitian symmetric domain $X$.
Let $Z \subset \cX \subset \CC^n$ be a complex analytic subvariety and $W \subset Z$ a
semi-algebraic set. There exists an
irreducible complex algebraic subvariety $Y \subset \CC^n$ such that
$$
W \subset Y \cap X \subset Z
$$

\end{lem}

\begin{proof}
This is a consequence of the proof of \cite[lemma 4.1]{PT}.
\end{proof}

\sspace
\noindent Bruno Klingler : Universit\'e Paris-Diderot (Institut de
Math\'ematiques de Jussieu-PRG, Paris) and IUF.

\noindent email : \texttt{klingler@math.jussieu.fr}.

\sspace
\noindent
Emmanuel Ullmo : Universit\'e Paris-Sud.

\noindent email:  \texttt{ullmo@math.u-psud.fr}

\sspace
\noindent
Andrei Yafaev : University College London, Department of Mathematics.

\noindent
email : \texttt{yafaev@math.ucl.ac.uk}
\end{document}